\documentclass[12pt,a4paper]{amsart}
\usepackage{amsmath,amssymb}
\usepackage{amsthm,graphicx, color}
\usepackage[T1]{fontenc}
\usepackage{float}

\newcommand{\calT}{\mathcal{T}}

\newcommand{\LHS}{\text{LHS}}

\newcommand{\ZZ}{\mathbb{Z}}

\newcommand{\MM}{\mathbb{M}}

\newcommand{\ZZp}{\ZZ'}

\DeclareMathOperator{\sgn}{sgn}

\newcommand{\circW}{W^\dagger}
\newcommand{\tmax}{t\!\!-\!\!\max}

\newtheorem{thm}{Theorem}
\newtheorem{lemma}[thm]{Lemma}

\newtheorem{coroll}[thm]{Corollary}

\newcommand{\setP}{\mathcal{P}}
\newcommand{\setH}{\mathcal{H}}

\newcommand{\smallsquare}{{\text{\scalebox{0.7}{$\square$}}}}
\newcommand{\res}{\text{res}}

\title{A multiset hook length formula and some applications} 
\date{May 9, 2011}
\author{Paul-Olivier Dehaye}
\address{Department of Mathematics\\
ETH Z\"urich\\
R\"amistrasse 101\\
8092 Z\"urich\\
Switzerland}
\email{pdehaye@math.ethz.ch}
\thanks{{The authors wish to thank the Centre National de la Recherche Scientifique, France and the Forschungsintitut f\"ur Mathematik, ETH Z\"urich, Switzerland for making this collaboration possible. }}


\author{Guo-Niu HAN}
\address{Institut de Recherche Math\'ematique Avanc\'ee\\
Universit\'e de Strasbourg et CNRS\\
7 rue Ren\'e-Descartes\\
 67084 Strasbourg\\
France}
\email{guoniu.han@unistra.fr}

\begin{document}
\begin{abstract}
A multiset hook length formula for integer partitions
is established by using combinatorial manipulation. 
As special cases,
we rederive three hook length formulas,
two of them	 obtained by  Nekrasov-Okounkov,
the third one by Iqbal, Nazir, Raza and Saleem,
who have made use of the cyclic symmetry of the topological vertex.
A multiset hook-content formula is also proved.
\end{abstract}

\maketitle

\section{Introduction} 

Recently, an elementary proof of the 
Nekrasov-Okounkov hook length formula \cite{NO} was given by the second author in \cite{HanNO},  using the Macdonald identities for $A_t$ (see \cite{macdonalddedekind}).
A crucial step of that proof is the construction of a 
bijection between
$t$-cores and integer vectors satisfying some additional properties.
Several further papers related to the Nekrasov-Okounkov formula have been 
published. See, \emph{e.g.},  \cite{Westbury,CLPS,CW,CKW, IKS, StanleyHook,Ols,INRS}.

In the present paper we again take up the study of the Nekrasov-Okounkov formula
and obtain several results in the following directions:
(1) The bijection between $t$-cores and integer vectors is constructed
for any positive integer~$t$, while in \cite{HanNO}, $t$ had to be an odd positive integer;
(2) That bijection is shown to satisfy a multiset hook length formula
(Theorem \ref{thm.vcoding})
with a functional parameter $\tau$ by using a geometric model,
		 called ``exploded tableau''. The result in \cite{HanNO} corresponds to the special case $\tau(x)=x$.
(3) A multiset hook length formula provides another special case when
taking $\tau=\sin$, namely Theorem \ref{thm.sin.Ht}. 
(4) Three hook length formulas are derived (Corollaries \ref{thm.NO} and \ref{thm.sin}, Theorem~\ref{thm.sin.H1}), the first two
 previously obtained by  Nekrasov-Okounkov \cite{NO},
the third one by Iqbal, Nazir, Raza and Saleem \cite{INRS}.
(5) Theorem \ref{thm.sin.Ht} provides a unified formula for 
the Nekrasov-Okounkov formula and the classical Jacobi triple product identity
(\cite[p.21]{Andrews}, \cite[p.20]{Knuth3}). This solves Problem~6.4 in 
 \cite{HanConj}.
(6) A multiset hook-content hook length formula is also given in Section \ref{sec.content}.

The basic notions needed here can be found in \cite[p.1]{macdonald},
\cite[p.287]{stanley}, \cite[p.1]{lascoux}, \cite[p.59]{Knuth3},
and \cite[p.1]{Andrews}).
A {\it partition}~$\lambda$ of {\it size} $n$ and of {\it length} $\ell$
is a sequence of positive 
integers $\lambda=(\lambda_1,\cdots, \lambda_\ell)$ such that 
$\lambda_1\geq \lambda_2 \geq \cdots \geq \lambda_\ell>0$ and
$n=\lambda_1+ \lambda_2+\cdots+ \lambda_\ell$. 
We write $n=|\lambda|$, $\ell(\lambda)=\ell$ and $\lambda_i=0$ for $i\geq \ell+1$.
The set of all partitions of size $n$
is denoted by $\setP(n)$. 
The set of all partitions is denoted by~$\setP$,
so that $\setP=\bigcup_{n\geq 0} \setP(n)$.
The {\it hook length multiset} of $\lambda$, denoted by $\setH(\lambda)$,
is the multiset of all hook lengths of $\lambda$.
Let $t$ be a positive integer. 
We write
${\setH}_t(\lambda)=\{h\mid h\in\setH(\lambda), h\equiv 0 \pmod t\}$.
A partition $\lambda$ is a {\it $t$-core} if 
$\setH_t(\lambda)=\emptyset$ (see 
\cite[p.69, p.612]{Knuth3}, \cite[p.468]{stanley}).
For example, $\lambda=(6,3,3,2)$ is a partition of size 14 and of length 4.
We have $\setH(\lambda)=\{2,1,4,3,1,5,4,2,9,8,6,\discretionary{}{}{} 3, 2, 1\}$
and $\setH_2(\lambda)=\{2,4,4,2,8,6,2\}$ (see also \cite{HanNO}).

Let $t$ be a positive integer and
$t_0 = 0$ (resp. $t_0=1/2$) if $t$ is odd (resp.~even).
Consider the set of (half-) integer $\ZZp = t_0+\ZZ$. 
Each vector of integers ${\vec V}=
(v_0, v_1, \ldots, v_{t-1})\in \ZZp^t$ is called 
{\it $V_t$-coding} if the following conditions hold:
(i) $\{	v_i-i\mod t : i=0, \cdots,t-1\}$ is equal to 
$t_0+\{0,1,\cdots, t-1\}$, 
(ii) $v_0+v_1+\cdots + v_{t-1}=0$,
(iii) $v_0 > v_1 > \cdots > v_{t-1}$.

\begin{thm}
\label{thm.vcoding} Let $t$ be a positive integer and
$\tau: \ZZ \rightarrow F$  
be any weight function from $\ZZ$ to a field $F$.
Then, 
there is a bijection 
$\phi_t: \lambda\mapsto {\vec V}=(v_0, v_1, \ldots, v_{t-1})$
from $t$-cores onto $V_t$-codings 
such that
\begin{eqnarray}
\label{eqn.vcoding.size}
\label{eqn.sum.square}
|\lambda|=\frac{1}{2t}(v_0^2+v_1^2+\cdots +v_{t-1}^2) - \frac{t^2-1}{24}
\end{eqnarray}
and
\begin{eqnarray}
\label{eqn.vcoding}
\prod_{h \in \setH(\lambda)} {\frac{\tau(h-t) \tau(h+t) }{\tau(h)^2}}
&=& 
\prod_{i=1}^{t-1} {\frac{\tau(-i)^{\beta_i(\lambda)}}{\tau(i)^{\beta_i(\lambda)+t-i}}}
\prod_{0\leq i<j\leq t-1} \tau(v_i-v_j),
\label{eqn.tau.hook}
\end{eqnarray}
where $\beta_i(\lambda) = \#\{\square \in \lambda : h(\square) = t-i\}$.
\end{thm}

The proof of Theorem \ref{thm.vcoding} is given in Section 
\ref{sec.thm.vcoding}.
With the weight function $\tau=\sin$, 
 an odd function, we get the specialization stated in the next theorem.
Its proof is given in  Section \ref{sec.sin}.

\begin{thm}\label{thm.sin.Ht}
For any positive integer~$r$ and any complex numbers $z,t$, we have
\begin{multline}\label{eqn.sin.Ht}
 \sum_\lambda q^{|\lambda|}\prod_{h\in \setH_r(\lambda)}
	 \left(1-{\frac{ \sin^2(tz) }{ \sin^2(hz) }} \right)  \\
=
\exp\sum_{k=1}^{\infty} \left({\frac{q^k}{k(1-q^k)}}  
-\frac{rq^{rk}}{k(1-q^{rk})}\frac{\sin^2(tkz)}{\sin^2(rkz)}\right).
\end{multline}
\end{thm}

Some specializations of Equation~(\ref{eqn.sin.Ht}) are given
in Section \ref{sec.specializations}.


\section{Exploded tableau}  

\label{sec.bijection}
With each partition $\lambda = (\lambda_1, \lambda_2,\cdots,\lambda_\ell)$ 
and 
each positive integer~$t$ we associate  several sets of (half-)integers. 
All these concepts will be illustrated for the case $\lambda = (8, 4, 3, 2, 2, 1)$ and $t=5$ 
(See Table~\ref{table5}). 
Note that this case is special, as $\lambda$ is itself a $t$-core, but this property will be assumed most of the time.  

The $W$-\emph{set} of $\lambda$ is a translation of the shifted parts, 
defined to be the set of
all integers of the form $\lambda_i -i + (t+1)/2$ for $i \in \mathbb{N}\setminus 0$ (the partition $\lambda$ is 
viewed as an infinite non-increasing sequence trailing with zeroes). We denote this set by $W(\lambda)$. It is immediate that $W(\lambda) \subset \ZZp$. It is also clear  that there exists a smallest (half-)integral $M= M(\lambda)$ and a largest {(half-)integral} $m=m(\lambda)$ such that $\{m,m-1, \cdots \} \subseteq W(\lambda) \subseteq \{M,M-1,\cdots \}$.  

\begin{table}
\begin{small}
\begin{tabular}{c|c}
$\lambda $ & $  (8, 4, 3, 2, 2, 1)$\\
$W(\lambda)$ & $ \left\{10, 5, 3, 1, 0, -2, -4, -5, -6,-7,-8,-9,-10, -11,\cdots\right\}$\\
$V(\lambda) $ & $ \left\{10, 3, 1, -6, -8\right\}$\\
$\circW(\lambda)  $ & $ \left\{5,  0, -2, -4, -5,  -7, -9, -10, -11,\cdots\right\}$\\
$M(\lambda) $ & $ 10$\\
$m(\lambda)  $ & $ -4$\\
$C(\lambda)  $ & $ \{9, 8, 7, 6, 4, 2, -1, -3\}$\\
\\
\hline
\\
$\lambda^t $ & $ (6, 5, 3, 2, 1, 1, 1, 1)$\\
$W_2(\lambda)$ & $ \left\{8, 6, 3, 1, -1, -2, -3, -4, -6, -7, -8,-9,-10,-11,-12,\cdots\right\}$\\
$V_2(\lambda) $ & $ \left\{8, 6, -1, -3, -10\right\}$\\
$\circW_2(\lambda)  $ & $ \left\{3, 1, -2, -4, -6, -7, -8, -9, -11, -12, \cdots\right\}$\\
$M_2(\lambda) $ & $ 8$\\
$m_2(\lambda)  $ & $ -6$\\
$C_2(\lambda)  $ & $ \{7, 5, 4, 2, 0, -5\}$\\
\end{tabular}
\end{small}
\medskip
\caption{The example $\lambda = (8, 4, 3, 2, 2, 1)$ with $t=5$. Note that this also gives $W_1(\lambda), V_1(\lambda),$ etc since $W(\lambda) = W_1(\lambda),$ etc. \label{table5}}
\end{table}

We say that an element $x$ in a set $X$ is {\it $t$-maximal} if it is the largest in its congruence class modulo $t$. 
If $t$ is even, we have $W(\lambda) \subset \frac{\mathbb{Z}}{2}$. By ``congruence classes mod $t$'', we then mean the congruence classes mod $t$ of $1/2,3/2,\cdots,t-1/2$.
The set of $t$-maximal elements
is denoted by $\tmax(X)$.
In the cases further considered, congruence classes will always contain an element, so no maximum will ever be taken over an empty set. It is then clear that $|\tmax(X)|=t$. 

\medskip

We define the $V$-\emph{set} $V(\lambda)$ of $\lambda$ by 
$V(\lambda) := \tmax(W(\lambda))$. It is easily seen from the definition of $m(\lambda)$ that no congruence class modulo $t$ can be empty. We also set ${\circW} (\lambda) = W(\lambda) \setminus V(\lambda)$. If $V(\lambda)$ is 
sorted by decreasing order,
we get a $V_t$-coding (as proved in Equation~(\ref{eqn.sum.zero})), 
that will be  denoted by  ${\vec V}(\lambda)=\phi_t(\lambda)$.
Thus, the bijection $\phi_t$ required in Theorem \ref{thm.vcoding} is constructed.
\medskip

We also define the complementary set
$C(\lambda) := \{M,M-1, \cdots \} \setminus W(\lambda)$,
so that the disjoint union
$\circW(\lambda) \cup V(\lambda) \cup C(\lambda)$ is equal to  $\{M,M-1, \cdots \}$.
Note that $m(\lambda) = \min C(\lambda) -1$. 

The invariants previously defined, such as 
$V(\lambda), W(\lambda), \ldots$ will also be given the subscript ``1'', as in $V_1(\lambda), W_1(\lambda), \ldots$
The invariants attached to the conjugate partition $\lambda^t$, such as
$V(\lambda^t), W(\lambda^t), \ldots$ will then be written $V_2(\lambda), W_2(\lambda), \ldots$

\medskip

The \emph{exploded diagram} 
of a partition $\lambda$, which we now define, is a basic tool in the construction.
The reader is referred to Figure~\ref{fig.full.exploded} for an example when $t$ is odd, and Figure~\ref{fig.full.exploded.even} when $t$ is even. We start with a two-dimensional lattice $\ZZ'\times \ZZ' \subset \mathbb{R}^2$, and
add a $1\times1$ box in each position  $ (\lambda_i-i+\frac{t+1}{2}, \lambda^t_j-j+\frac{t+1}{2})$ of the lattice for every $i, j \in \mathbb{N}_{>0}$.
This means that there is one box in each element of $W_1(\lambda) \times W_2(\lambda)$. 
In contrast to the classical Ferrers diagram, the exploded diagram is thus infinite.  
The \emph{entry} of each box in the exploded diagram is defined to be the the sum of the two coordinates of the box. 
When the entry is explicitly written on each box, we shall speak of an  \emph{exploded tableau}.

Boxes of constant entry line up on anti-diagonals. We use this fact to group boxes into different sets. 
Let $\Delta$ (resp. $\Gamma^+$, resp. $\Gamma^-$) be the set of all boxes with entries in the range $(t, \infty)$ 
(resp. $(0,t)$, resp. $(-t,0)$).
The set $\Delta$ corresponds to the boxes of $\lambda$ in the classical Ferrers diagram (which are {\it shaded} in Fig.~1). 
In addition, if $(x,y) \in\Delta$ corresponds to $\square \in \lambda$, its entry $x+y$ in the exploded tableau is equal to $h_\square +t$. 
The entries lower than $t$ correspond to outside hooks, and there are thus no box with entry exactly $t$.

\begin{figure}
\centering
\begin{center}
\includegraphics[width=0.9\textwidth]{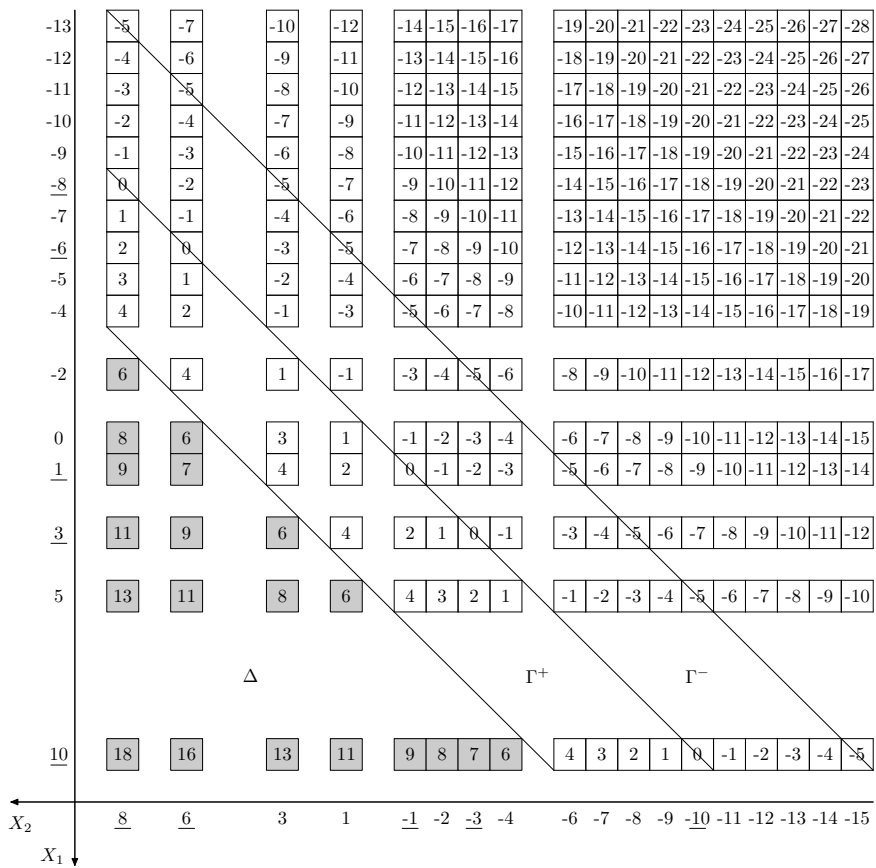}
\end{center}
\caption{The exploded tableau of the partition $\lambda = (8, 4, 3, 2, 2, 1)$, with $t=5$. The partition (shaded boxes)  
appears in an orientation similar to the French orientation for Ferrers diagrams. 
Therefore, axes are reversed and switched. 
The $W$-sets of $\lambda$ and $\lambda^t$ serve as coordinates and  the $V$-sets are underlined. For instance, 
consider the third box of the second row of the classical Ferrers diagram of $\lambda$, \emph{i.e.}~the second-to-last box on that row. This box now ends up at coordinates $( 4-2+\frac{5+1}{2},3-3+\frac{5+1}{2}) = (5,3)$, carrying the entry $5+3=8$.  
Three diagonal lines separate $\Delta, \Gamma^+$ and $\Gamma^-$. \label{fig.full.exploded}} 
\end{figure}

Given a set $X$, we write $-X$ for the set of opposites of elements of $X$. 
In the special case of a $t$-core, many of the invariants we just defined are nicely related.
\begin{lemma}
If $\lambda$ is a $t$-core, then
\begin{eqnarray}
W(\lambda) &=& \bigcup_{a \in V(\lambda)} \left(a+ t\mathbb{N}\right) \label{eqn.decomposition}\\
V_1(\lambda) &=& W_1(\lambda) \cap -W_2(\lambda)  \label{eqn.opposite.cap}\\
V_2(\lambda) &=& -V_1(\lambda) \label{eqn.negative}\\
\circW_1(\lambda) &=& -C_2(\lambda) \cup \left\{-M_2(\lambda)-1, -M_2(\lambda)-2,\cdots\right\}\label{eqn.complementary}\\
\sum_{v \in V(\lambda)} v &=& 0 \label{eqn.sum.zero}.
\end{eqnarray}
\end{lemma}
\begin{proof}
Equation~(\ref{eqn.decomposition}) follows from the definitions: if we had  $v  {} \notin {}  W(\lambda)$ with $v+t \in W(\lambda)$,
there would exist 	a hook of length $t$ on the row that generates $v+t$, and $\lambda$ would not be a $t$-core.
It is a classical lemma in combinatorics that for any partition $\lambda$,  the two sets $\left\{\lambda_i -i +1/2 : i \ge 1 \right\}$ and $-\left\{\lambda^t_i -i +1/2 : i \ge 1 \right\}$ are disjoint and that their union is $\mathbb{Z}+1/2$. 
The sets $W_1(\lambda)$ and $-W_2(\lambda)$ are merely translates of these two classical sets.
In light of Equation~(\ref{eqn.decomposition}) the sets
$W_1$ and $-W_2$ intersect in just one point for each congruence class mod~$t$. 
It is easy to show that the set of all those points is actually $V_1(\lambda)$ or $-V_2(\lambda)$ (this is Equations~(\ref{eqn.opposite.cap}) and (\ref{eqn.negative})). 

Equation~(\ref{eqn.complementary}) is a quick consequence of the previous three.
A proof of identity~(\ref{eqn.sum.zero}) by using the
Durfee square of the partition $\lambda$
can be found in \cite{Garvan}.
\end{proof}
\section{Proof of Theorem~\ref{thm.vcoding}} 
\label{sec.thm.vcoding}
Throughout the proof we will use the example of $\lambda =(8,4,3,2,2,1)$ and $t=5$, as illustrated
in Fig.~\ref{fig.simplified.exploded} (for $t$ odd) and Fig.~\ref{fig.full.exploded.even} (for $t$ even).
When $t$ is even, both coordinates are half-integers. The entries are still integral and the argument carries through identically.  
Let $\calT_{(a,b)}: \ZZp \rightarrow \ZZp$ denote the translation defined by $\calT_{(a,b)}(x,y) = (x+a,y+b)$. 
We now need the following easy results:
\begin{eqnarray}
\mathbf{1}_{W_1\times \circW_2} +\mathbf{1}_{\circW_1\times W_2} &=& \mathbf{1}_{W_1\times W_2\setminus V_1\times V_2}+\mathbf{1}_{\circW_1\times \circW_2}
\label{eqn.T1}\\
\calT_{(0,-t)}\left(\Delta\right) &=& (\Delta\cup\Gamma^+) \cap (W_1 \times \circW_2),\label{eqn.T2}\\
\calT_{(-t,0)}\left(\Delta\right) &=& (\Delta\cup\Gamma^+) \cap (\circW_1 \times W_2),\label{eqn.T3}\\
\calT_{(-t,-t)}\left(\Delta\right) &=& (\Delta\cup\Gamma^+\cup\Gamma^-) \cap (\circW_1 \times \circW_2).\label{eqn.T4}
\end{eqnarray}
The first one, where $\mathbf{1}$ is the indicator function, is completely trivial. For the second, assume $\square = (x,y) \in \lambda$, 
so that $x+y = h_\square +t$ and $x \in W_1$, $y\in W_2$. Then, its $(0,-t)$ translate, equal to $(x,y-t)$, has entry $x+y-t$ which is nonnegative. Also, $y-t$ is in $\circW_2$, since $y-t$ is not $t$-maximal. Finally, $y-t \in \circW_2$ is equivalent to $y \in W_2$. The other two identities follow similarly.
\medskip

\begin{proof}[Proof of Theorem \ref{thm.vcoding}]
For any set $B$ of boxes in the exploted tableau let 
$$
||B|| := \prod_{(x,y)\in B} \tau(x+y).
$$
The left-hand side of Equation~(\ref{eqn.vcoding}) can be written 
\begin{eqnarray}
\label{eqn.first.ratio}
\LHS=\frac{||\Delta|| \cdot || \calT_{(-t,-t)}\left(\Delta\right)||} {
 ||	\calT_{(-t,0))}\left(\Delta\right)||\cdot || 
	 \calT_{(0,-t)}\left(\Delta\right) ||}.
\end{eqnarray}
Using relations (\ref{eqn.T1}-\ref{eqn.T4}) we can rewrite Expression~(\ref{eqn.first.ratio}) as  
\begin{eqnarray}
\label{eqn.tau.entry.2}
\LHS=\frac{||\Delta \cap (V_1\times V_2) ||\cdot || \Gamma^- \cap  (\circW_1 \times \circW_2 )||}
{|| \Gamma^+ \setminus (V_1\times V_2)||}.
\end{eqnarray}
This information is summarized graphically in Figure~\ref{fig.simplified.exploded} for our running example 
(the numerator is the product of the entries in squares containing a value, 
 while the denominator is
the product of the entries in circles). At this point the reader is encouraged to consider Figure~\ref{fig.simplified.exploded} to anticipate the next step: we aim to ``fold'' the boxes in the region $\Gamma^+$ and interleave them with boxes in the region $\Gamma^-$. 

\begin{figure}
\begin{center}
\includegraphics[width=0.9\textwidth]{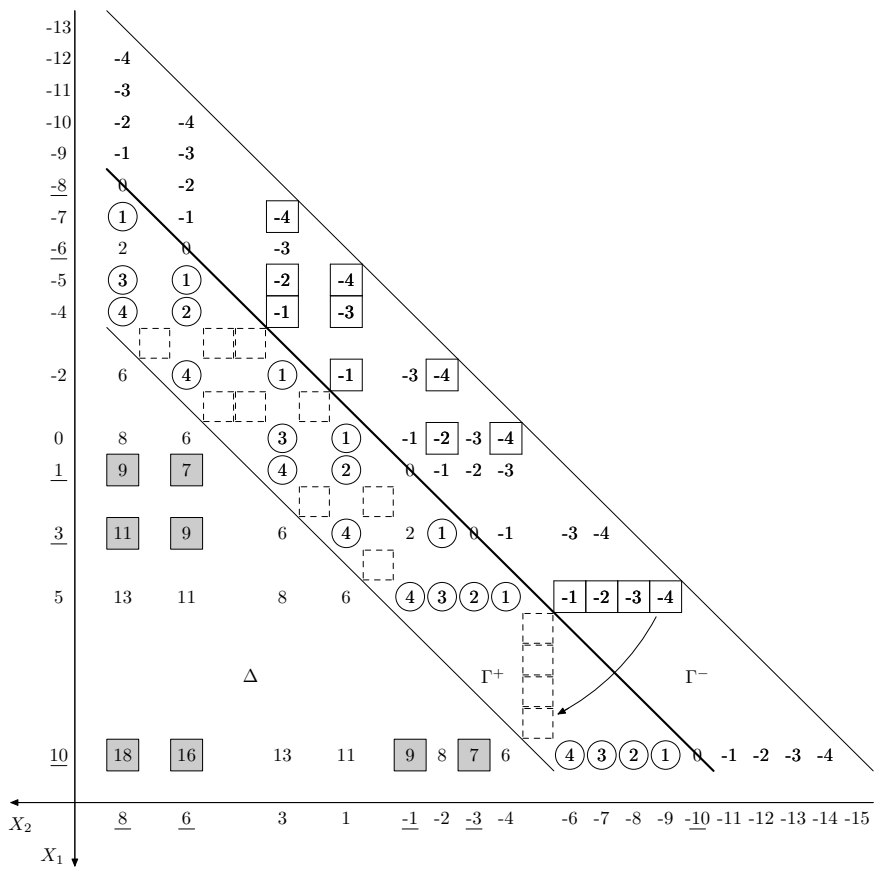}
\end{center}
\caption{The product given in Equation~(\ref{eqn.tau.entry.2}) marked in a graphical way. 
First, $B$ is the set of all boxes (shaded or not) located in $\Delta$ with
	entries in the range $[6,9]$.
	The set $\calT_{(-t,-t)}\left(B\right)$ is materialized by all the squares appearing in $\Gamma^-$.
	The set $\calT_{( -t,0)}\left(B\right)\cup \calT_{(0,-t)}\left(B\right)$ consists of all the circles appearing in $\Gamma^+$.
The middle diagonal indicates the diagonal used to fold the boxes of $\Gamma^-$, while the dashed squares show the locations where those boxes end up. One example is indicated with the arrow.
\label{fig.simplified.exploded}}
\end{figure}

Consider the map $\mathcal{F}: \ZZp\times \ZZp \rightarrow \ZZp\times\ZZp$, sending $(x,y)$ to $(-y,-x)$. 
By Equation~(\ref{eqn.complementary}), this map is a bijection between 
$\Gamma^- \cap  (\circW_1 \times \circW_2)$ and $\Gamma^+ \cap (C_1\times C_2)$. It also merely changes the sign of $x+y$. Hence,
\begin{eqnarray}
\label{eqn.tau.entry.3}
\qquad || \Gamma^- \cap  (\circW_1 \times \circW_2)|| =\prod_{i=1}^{t-1} {\left(\frac{\tau(-i)}{\tau(i)}\right)^{\beta_i(\lambda)}}
||\Gamma^+ \cap (C_1\times C_2)||.
\end{eqnarray}

We now claim that 
\begin{eqnarray}
\frac
{||\Gamma^+ \cap ( W_1 \times (M_2 - \mathbb{N}))||}
{||\Gamma^+\cap (\ZZp \times C_2)||}
 &=&  \prod_{i=1}^{t-1} \tau(i)^{t-i}.
 \label{eqn.triangle}
\end{eqnarray}
The proof of this claim works by observing that in the denominator, the product over all 
boxes in a given column is \emph{always} of the form $\prod_{i=1}^{t-1} \tau(i)$. 
For the numerator, the product of all boxes in a given row is also of the form $\prod_{i=1}^{t-1} \tau(i)$,
except for the highest $t-1$ rows, which end up producing the right-hand side. 
We are left to count the multiplicities of  the full product  $\prod_{i=1}^{t-1} \tau(i)$ in numerator and denominator, and get $ |W_1 \cap [-M_2,\infty)| -t = |V_1\cup -C_2|-t=|C_2|  $
by Equation~(\ref{eqn.complementary}).
Thereby proving Equation~(\ref{eqn.triangle}). Hence
\begin{equation}
\frac {||\Gamma^+ \cap ( W_1 \times W_2)||} {||\Gamma^+\cap (C_1 \times C_2)||}
=
\frac {||\Gamma^+ \cap ( W_1 \times (W_2\cup C_2))||} {||\Gamma^+\cap ((W_1\cup C_1) \times C_2)||}
 =  \prod_{i=1}^{t-1} \tau(i)^{t-i}.
 \label{eqn.triangle.full}
\end{equation}
By Equations (\ref{eqn.tau.entry.2}), (\ref{eqn.tau.entry.3}), (\ref{eqn.triangle.full}) we derive 
\begin{eqnarray*}
\LHS&=&\prod_{i=1}^{t-1} {\left(\frac{\tau(-i)}{\tau(i)}\right)^{\beta_i(\lambda)}}
\frac{||\Delta  \cap V_1\times V_2 ||\cdot||  \Gamma^+ \cap V_1\times V_2|| }{\prod_{i=1}^{t-1} \tau(i)^{t-i}}  \\ 
&=&
\prod_{i=1}^{t-1} {\frac{\tau(-i)^{\beta_i(\lambda)}}{\tau(i)^{\beta_i(\lambda)+t-i}}}
\prod\limits_{\substack{(x,y) \in V_1\times V_2 \\ x+y \ge 1 }} \tau(x+y)   \\
&=&
\prod_{i=1}^{t-1} {\frac{\tau(-i)^{\beta_i(\lambda)}}{\tau(i)^{\beta_i(\lambda)+t-i}}}
\prod\limits_{0\le i<j\le t-1} \tau(v_i-v_j). 
\end{eqnarray*}
The last equality follows from Equation~(\ref{eqn.negative}). This equals the right-hand side of Equation~(\ref{eqn.vcoding}). 

\medskip

We still have to prove Equation~(\ref{eqn.vcoding.size}). For this, we conveniently rely on Equation~(\ref{eqn.vcoding}) with 
the special weight function
$\tau(k) = 1+ z k^2$. 
By considering the coefficient of $z$ on both sides, we get
\begin{multline}
2 |\lambda| t^2 = \sum_{h \in \mathcal{H}(\lambda)} \left((h-t)^2+(h+t)^2-2h^2\right) \overset{\text{Eq.(\ref{eqn.vcoding})}}{=} \\
\left(-\sum_{k=1}^{t-1} k^2(t-k) \right)+ \sum_{0 \le i<j\le t-1} (v_i-v_j)^2 = \\
\left(-\frac{1}{12} t^2 (t^2-1) \right)+ \left( t \sum_{i=0}^{t-1} v_i^2 + \left(\sum_{i=0}^{t-1}v_i\right)^2\right),
\end{multline}
which implies the result, thanks to Equation~(\ref{eqn.sum.zero}).
\end{proof}

When the weight function $\tau$ is either even or odd, the right-hand side of Equation (\ref{eqn.vcoding}) can be simplified.
In the next Corollary we assume that ${\bf u}=(u_0, u_1, \ldots, u_{t-1})$ is the vector obtained by sorting the $V_t$-coding ${\vec V}$  
according to the congruence classes, \emph{i.e.},
$u_i\equiv i+t_0 \pmod t$  for $0\leq i\leq t-1$.

\begin{coroll}
\label{thm.vcoding.evenodd} 
We have
\begin{eqnarray}
\label{eqn.vcoding.evenodd}
\prod_{h \in \setH(\lambda)} {\frac{\tau(h-t) \tau(h+t) }{\tau(h)^2}}
&=& \frac{C}{\prod_{k=1}^{t-1} \tau(k)^{t-k}} \prod_{0\leq i<j\leq t-1} \tau(u_i-u_j),
\label{eqn.tau.hook.evenodd}
\end{eqnarray}
with
$$C =
\left\{
	\begin{array}{ll}
		-1  & \text{if $t \equiv 3 \mod 4$ while $\tau$ is odd,}\\
		1 & \text{otherwise.}
	\end{array}
\right.
$$
\end{coroll}

\begin{proof}
We need to split the boxes in $\Gamma^-$ 
according to the congruence classes for the coordinates of the boxes. 
We know that $W(\lambda) \subset t_0+\mathbb{Z}$. It can be shown that for all $i,j \in  \{0,1,2\cdots,t-1\}$,
\begin{multline*}
\left|\left\{ (x,y) \in \Gamma^-: x \equiv i+t_0\mod t, y  \equiv j+t_0 \mod t \right\}\right| 
\\ = \max\left(0,\left\lfloor \frac{u_i-u_j}{t}\right\rfloor\right) =
 \max(0,k_i-k_j -\delta_{i<j} )
\end{multline*}
if ${\bf u} = (u_i) = (i + t_0+ t\cdot k_i)_{i=0}^{t-1}$. 
Therefore, 
\begin{multline}
\Gamma^- = \left|\left\{ \square \in \lambda : h_\square < t \right\}\right| = \\  \sum_{\substack{i,j\\ u_i > u_j}}  ((k_i-k_j)-\delta_{i<j} )
\equiv \sum_{\substack{i,j\\ u_i>u_j}} ((k_i+k_j)+\delta_{i<j} ) \mod 2 \\
\equiv (t-1) \left(\sum_{i=0}^{t-1} k_i \right) + \frac{t(t-1)}{2} +\sgn \prod_{i<j}(u_i-u_j)\mod 2.
\label{eqn.sgn}
\end{multline}
We know (by Equation~(\ref{eqn.sum.zero})) that
$
0 = \sum_{i=0}^{t-1} u_i = \sum_{i=0}^{t-1} \left(t_0 + i + k_i t \right),
$
which gives $-\sum  k_i = \frac{t-1}{2}+t_0$. 
Together with Equation~(\ref{eqn.sgn}), this easily gives $C=(-1)^{(t_0-{1/2})(t-1)}$, which is merely a restatement of Equation~(\ref{eqn.vcoding.evenodd}).
\end{proof}


{
\renewcommand{\tabcolsep}{2.5pt}
\renewcommand{\arraystretch}{1.4}
\begin{table}[ht]
\begin{small}
\begin{tabular}{c|c}
$\lambda$ &  (8, 5, 4, 1, 1, 1)\\
$W(\lambda)$ & $\left\{\frac{21}{2}, \frac{13}{2}, \frac{9}{2}, \frac{1}{2}, -\frac{1}{2}, -\frac{3}{2}, -\frac{7}{2}, -\frac{9}{2}, -\frac{11}{2},-\frac{13}{2},-\frac{15}{2},-\frac{17}{2},-\frac{19}{2},-\frac{21}{2},
\cdots\right\}$\\
$V(\lambda)$ & $\left\{\frac{21}{2}, \frac{13}{2}, -\frac{1}{2}, -\frac{7}{2}, -\frac{9}{2}, -\frac{17}{2}\right\}$\\
$\circW(\lambda)$  & $\left\{\frac{9}{2}, \frac{1}{2}, -\frac{3}{2}, -\frac{11}{2}, -\frac{13}{2}, -\frac{15}{2}, -\frac{19}{2}, -\frac{21}{2},  \cdots\right\}$\\
$M(\lambda)$ & $\frac{21}{2}$\\
$m(\lambda)$  & $-\frac{7}{2}$\\
$C(\lambda)$  & $\{\frac{19}{2}, \frac{17}{2}, \frac{15}{2}, \frac{11}{2}, \frac{7}{2}, \frac{5}{2}, \frac{3}{2}, -\frac{5}{2}\}$\\
\hline
$\lambda^t$ & $(6, 3, 3, 3, 2, 1, 1, 1)$\\
$W_2(\lambda)$& $\left\{\frac{17}{2}, \frac{9}{2}, \frac{7}{2}, \frac{5}{2}, \frac{1}{2}, -\frac{3}{2}, -\frac{5}{2}, -\frac{7}{2}, -\frac{11}{2}, -\frac{13}{2}, -\frac{15}{2},-\frac{17}{2},-\frac{19}{2},-\frac{21}{2},-\frac{23}{2},
\cdots\right\}$\\
$V_2(\lambda)$ &$ \left\{\frac{17}{2}, \frac{ 9}{2}, \frac{7}{2}, \frac{1}{2}, -\frac{13}{2}, -\frac{21}{2}\right\}$\\
$\circW_2(\lambda) $ &$ \left\{\frac{5}{2}, -\frac{3}{2}, -\frac{5}{2}, 
-\frac{7}{2}, -\frac{11}{2}, -\frac{15}{2}, -\frac{17}{2}, -\frac{19}{2}, 
-\frac{23}{2}, 
\cdots\right\}$\\
$M_2(\lambda)$ &$ \frac{17}{2}$\\
$m_2(\lambda)$  & $-\frac{11}{2}$\\
$C_2(\lambda)$  &$ \{\frac{15}{2}, \frac{13}{2}, \frac{11}{2}, \frac{3}{2}, -\frac{1}{2}, -\frac{9}{2}\}$
\end{tabular}
\end{small}
\medskip
\caption{The example $\lambda = (8,5,4,1,1,1)$ with $t=6$.}
\end{table}
}

\begin{figure}[ht]
\centering
\includegraphics[width=0.9\textwidth]{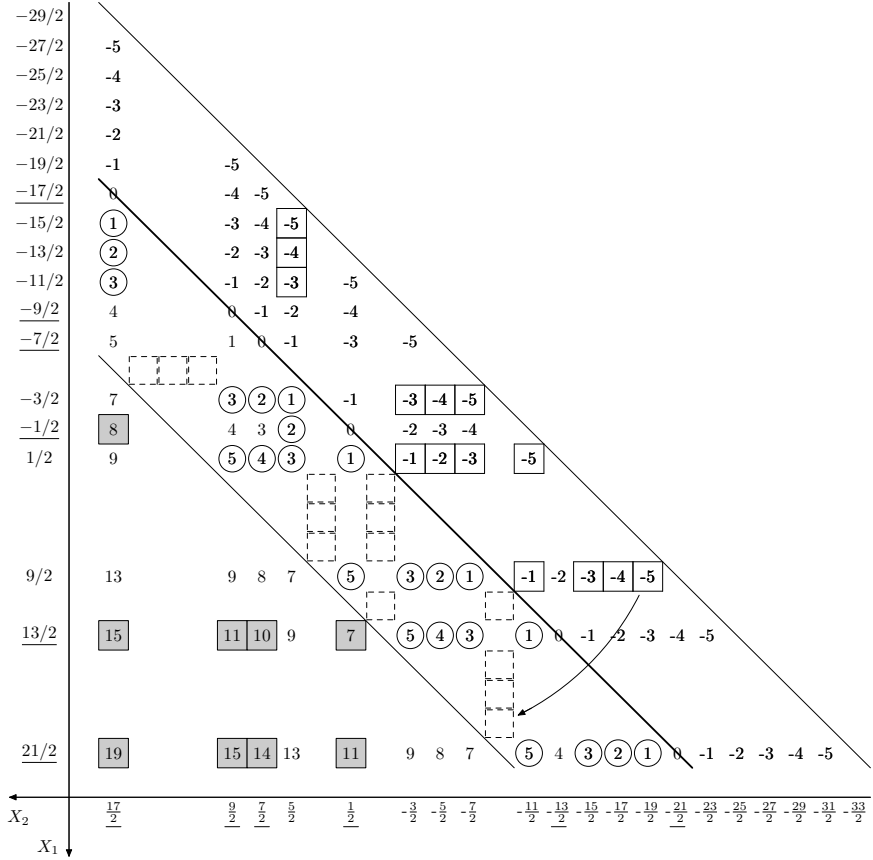}
\caption{Analog of Figure~\ref{fig.simplified.exploded}, but for 
 the partition	$\lambda = (8, 5, 4, 1, 1, 1)$, with $t=6$.}
\label{fig.full.exploded.even}
\end{figure}


\section{Specializations}  
\label{sec.specializations}
We derive some specializations of the multiset hook length formula 
(Theorem \ref{thm.vcoding}).
The simplest non-trivial example is the case where 
the weight function $\tau(x)=x$. Theorem \ref{thm.vcoding}
is then equivalent to Theorem 1.1 in \cite{HanNO}, which provides a combinatorial proof
of a hook length formula due to Nekrasov and Okounkov \cite[formula (6.12)]{NO}
(see also Equation~(\ref{eqn.NO}))
by using the Macdonald identities for 
$A_t$ \cite{macdonalddedekind}.
When we take $\tau=\sin$, which is an odd function,
thanks to some properties of the function $\sin$ (see Lemma \ref{thm.sin.properties}), we derive Theorem \ref{thm.sin.Ht} in Section \ref{sec.sin}.

If  $r$ equals 1, 
we obtain the following hook length formula.
\begin{thm}[$r=1$]\label{thm.sin.H1}
For any complex numbers $z$ and $t$, we have
\begin{equation}\label{eqn.sin.H1}
 \sum_\lambda q^{|\lambda|}\prod_{h\in \setH(\lambda)}
	 \Bigl(1-{\frac{ \sin^2(tz) }{ \sin^2(hz) }} \Bigr)  
=
\exp \sum_{k=1}^{\infty} {\frac{q^k}{k(1-q^k)}}  
\Bigl(1-\frac{\sin^2(tkz)}{\sin^2(kz)}\Bigr).
\end{equation}
\end{thm}

It can be shown that formula (\ref{eqn.sin.H1}) is equivalent to
the combination of the two identities (2.4) and (2.7) in the
paper written by Iqbal et al. \cite{INRS}. 
Those authors have made use of the cyclic symmetry of the topological vertex \cite{AKMV, ORV}.
When $t=0$ in Theorem \ref{thm.sin.H1}, we obtain 
the  classical generating function for partitions.

\begin{coroll}[$r=1, t=0$]
We have
\begin{equation*}
 \sum_\lambda q^{|\lambda|}
=
\exp\left(\sum_{k=1}^{\infty} {\frac{q^k}{k(1-q^k)}}  
\right)
=\prod_{m=1}^{\infty}{\frac{1}{1-q^m}}.
\end{equation*}
\end{coroll}

Since
\begin{equation*}
\frac{\sin(az)}{\sin(bz)}=
\frac{a - a^3z^2/6 + \cdots}{b - b^3z^2/6 + \cdots},
\end{equation*}
Equation~(\ref{eqn.sin.H1}) becomes 
\begin{equation*}
 \sum_\lambda q^{|\lambda|}\prod_{h\in \setH(\lambda)}
	 \Bigl(1-\frac{ t^2}{h^2}\Bigr)
=
\exp\Bigl(\sum_k {\frac{q^k}{k(1-q^k)}}  (1-t^2)\Bigr).
\end{equation*}
when $z=0$. 
We also obtain the following hook formula due to Nekrasov and Okounkov \cite[Equation~(6.12)]{NO} (see also \cite{HanNO}):
\begin{coroll}[$r=1, z=0$]
\label{thm.NO}
For any complex number $\beta$ we have
\begin{equation}\label{eqn.NO}
 \sum_\lambda q^{|\lambda|}\prod_{h\in \setH(\lambda)}
	 \Bigl(1-\frac{ \beta}{h^2}\Bigr)
	=
\prod_m (1-q^m)^{\beta-1}.
\end{equation}
\end{coroll}


Let $e^{2itz}=s$ and $q=qs$  in Theorem \ref{thm.sin.H1}.
Equation~(\ref{eqn.sin.H1}) becomes 
\begin{multline}\label{eqn.poly}
\sum_\lambda q^{|\lambda|}\prod_{h\in \setH(\lambda)}
	 \Bigl(s+{\frac{s^2-2s+1 }{4\sin^2(hz)}} \Bigr)   \cr
\qquad \qquad =
\exp\Bigl(\sum_k {\frac{q^k}{k(1-s^kq^k)}}  
		\bigl(s^k+{\frac{s^{2k}-2s^k+1 }{4\sin^2(kz)}}\bigr)\Bigr).
\end{multline}
Letting $s=0$ yields
\begin{coroll}[$r=1, e^{2itz}=0$]\label{thm.sin} 
We have
\begin{equation}\label{eqn.sin}
 \sum_\lambda q^{|\lambda|}\prod_{h\in \setH(\lambda)}
	 \frac{1 }{ 4\sin^2(hz) } 
=
\exp\Bigl(\sum_k \frac{q^k}{4k\sin^2(kz)}\Bigr).
\end{equation}
\end{coroll}
Note that Equation~(\ref{eqn.sin}) has the following equivalent form:
\begin{equation*}
 \sum_\lambda q^{|\lambda|}\prod_{h\in \setH(\lambda)}
	 \frac{1 }{ 2-2\cos(hz) } 
=
\exp\Bigl(\sum_k \frac{q^k}{2k(1-\cos(kz))}\Bigr),
\end{equation*}
or, since $\sinh(x)=-i\sin(ix)$, 
\begin{equation}\label{eqn.NO725}
\sum_\lambda q^{|\lambda|}\prod_{h\in \setH(\lambda)}
	 \Bigl(-{\frac{1 }{4\sinh^2(hz)}} \Bigr)  
=
\exp\Bigl(\sum_k {\frac{q^k}{k}}  
		\bigl(-{\frac{1 }{4\sinh^2(kz)}}\bigr)\Bigr).
\end{equation}
Equation~(\ref{eqn.NO725}) and
Equation~(7.25) in \cite{NO} are the same. Minor typos are to be corrected
in the later paper. 

Let $s=-1$ in Equation~(\ref{eqn.poly}). We immediately have
\begin{multline}\label{eqn.poly.s-1}
\sum_\lambda q^{|\lambda|}\prod_{h\in \setH(\lambda)}
	 \Bigl(-1+{\frac{1 }{\sin^2(hz)}} \Bigr)   \cr
\qquad \qquad =
\exp\Bigl(\sum_k {\frac{q^k}{k(1-(-1)^kq^k)}}  
		\bigl((-1)^k+{\frac{2-2(-1)^k }{4\sin^2(kz)}}\bigr)\Bigr).
\end{multline}

\begin{coroll}[$r=1, e^{2itz}=-1$]\label{thm.cot}
We have
\begin{eqnarray*}
 & &\sum_\lambda q^{|\lambda|}\prod_{h\in \setH(\lambda)}
 \cot^2(zh) \cr
& &\qquad\qquad =
\exp\Bigl(\sum_{k\geq 1} {\Bigl(
		\frac{q^{2k-1}  \cot^2((2k-1)z) }{(2k-1)(1+q^{2k-1})}}  
		 + \frac{q^{2k}}{2k(1-q^{2k})}\Bigr)\Bigr).
\end{eqnarray*}
\end{coroll}

When $r=1$ and $t=2$,  Equation~\ref{eqn.sin.Ht} becomes
the Jacobi triple product identity. 
\begin{coroll}[$r=1, t=2$]\label{thm.jacobi}
We have
\begin{equation}
\prod_{n\geq 0} (1+a x^{n+1}) (1+x^n/a) (1-x^{n+1})
=\sum_{n=-\infty}^{+\infty} a^n x^{n(n+1)/2}.
\end{equation}
\end{coroll}


\section{Proof of Theorem \ref{thm.sin.Ht}} 
\label{sec.sin}
We use a  Macdonald identity for the proof of our theorem. 
Let $t$ be a positive integer.
Milne \cite{milne} and Le{\u\i}benzon\cite{leibenzon} provide the following version of 
Macdonald's identity \cite{macdonalddedekind} for the type $A_t$: 
let
$\MM_t=\{{\bf a}=(a_1, a_2, \ldots, a_{t})\in \ZZ^t  \mid 
a_1+a_2+\cdots a_{t} = 1 + 2 + \cdots + t\}$. For  $a\in \ZZ$ 
denote the residue of $a$ modulo $t$ by $\res_t a\in \ZZ/t\ZZ$.
For each sequence $(b_1, b_2, \ldots, b_t)$ of residues modulo $t$ 
define the number $\epsilon(b_1, b_2, \ldots, b_t)$ to be
equal to $0$ or $\pm 1$ according to the following rules:
if all of the $b_i$'s are different, \emph{i.e.}~$(b_1, b_2, \ldots, b_t)$
is a permutation of the sequence $(\res_t 1, \res_t 2, \ldots, \res_t t)$,
 then $\epsilon(b_1, b_2, \ldots, b_t)$ is the sign of the permutation;
otherwise, let $\epsilon(b_1, b_2, \ldots, b_t)=0$.
For each $ {\bf a}=(a_1, a_2, \ldots, a_{t})\in \ZZ^t$
let
$\epsilon({\bf a})= \epsilon(\res_t a_1, \res_t a_2, \ldots, \res_t a_t)$
and
$$\Omega({\bf a})=
\frac{1}{2t}\left(a_1^2+a_2^2+\cdots + a_t^2 -1^2 -2^2 -\cdots -t^2\right).$$ 
The Macdonald identity is then rewritten in the following form.
\begin{thm}\label{thm.macdonald}
For every $t\geq 2$ the identity
\begin{multline}\label{eqn.macdonald}
\prod_{m\geq 1} \Bigl((1-q^m)^{t-1}
\prod_{1\leq j<i\leq t} 
\Bigl(1-{\frac{x_i}{x_j}} q^{m-1}\Bigr)
\Bigl(1-{\frac{x_j}{x_i}} q^{m}\Bigr) \Bigr)\cr
=\sum_{{\bf a}\in \MM_t} \epsilon({\bf a}) q^{\Omega({\bf a})} x_1^{1-a_1}\cdots 
x_t^{t-a_t} 
\end{multline}
holds in the ring of formal power series in $q$ with coefficients from
the ring of Laurent polynomials in $x_1, x_2, \ldots, x_t$.
\end{thm}
Let $t=2t'+1$ be an odd integer.
The right-hand side of Equation~(\ref{eqn.macdonald}) reads
\begin{eqnarray*}
A(q)&=&\sum_{\substack{{\bf a}\in \MM_t \\ a_i\equiv i-1 \mod t}} \sum_{\sigma\in S_n}
\epsilon(\sigma) q^{\Omega({\bf a})} x_1^{1-\sigma(a_1)}\cdots x_t^{t-\sigma(a_t)} \\
&=& x_1^1 x_2^2 \cdots x_t^t \sum_{\substack{{\bf a}\in \MM_t \\ a_i\equiv i-1 \mod t }} q^{\Omega({\bf a})} \det\left(x_i^{-a_j}\right).
\end{eqnarray*}
Let ${\bf u}=(u_0, u_1, \ldots, u_{t-1})$ be a sequence defined by
$$u_i=a_{i+t'+2} - t'-1,$$ 
where $a_{t+j}=a_j$. Then
$u_i\equiv a_{i+t'+2} - t'-1 \equiv i+t'+2-1-t'-1 \pmod t$
and $\sum_{i=0}^{t-1} u_i=\sum_{i=1}^t -t(t'+1)=t(t-1)/2-t(t'+1)=0$, so that
the sorted vector ${\vec V}$ of ${\bf u}$ 
by decreasing order is a $V_t$-coding. 
Let $\lambda=\phi_t^{-1}({\vec V})$
where $\phi_t$ is defined in Section \ref{sec.bijection}.
Then
\begin{eqnarray}
\label{eqn.proof.size}
\Omega({\bf a})&=&\frac{1}{2t}\sum_{i=1}^t (a_i^2-i^2)\nonumber\\
&=&\frac{1}{2t}\sum_{i=1}^t ((u_{i-1}+t'+1)^2-i^2)\nonumber \\
&=&\frac{1}{2t}\sum_{i=0}^{t-1} u_{i} - \frac{t^2-1}{24} \overset{\text{Eq.(\ref{eqn.sum.square})}}{=} |\lambda|.
\end{eqnarray}
Hence (remember that $t$ is odd)
\begin{eqnarray*}
A(q) &=& x_1^1 x_2^2 \cdots x_t^t \sum_{{\bf u}} 
q^{|\lambda|} \det\left(x_i^{- (u_{j-1}+t'+1)}\right) \\
 &=& x_1^1 x_2^2 \cdots x_t^t (x_1x_2\cdots x_t)^{-t'-1}
 \sum_{{\bf u}} 
q^{|\lambda|} \det\left(x_i^{-u_{j-1}}\right).
\end{eqnarray*}

Let
$$
B(q)=\Bigl(\prod_{m\geq 1} (1-q^m)^{t-1}\Bigr)
\prod_{1\leq j<i\leq t} \prod_{m\geq 1}
\Bigl(1-{\frac{x_i}{x_j}} q^{m}\Bigr)
\Bigl(1-{\frac{x_j}{x_i}} q^{m}\Bigr).
$$
Then
\begin{equation}
B(q)= A(q)/A(0)=
\frac{ \sum_{{\bf u}} q^{|\lambda|} \det\left(x_i^{-u_{j-1}}\right) }{ \det\left(x_i^{-w_{j-1}}\right) },
\end{equation}
where ${\bf w}=(w_0, w_1,\ldots, w_{t-1})=(0,1,2,\ldots, t', -t', \ldots, -2, -1)$.
Let $x_i=e^{-2Iz(i-1)}$, with $I^2 = -1$. We have 
\begin{eqnarray}
 \det\left(x_i^{-u_{j-1}}\right) 
	&=&
 \det\left(e^{2Iz(i-1)u_{j-1}}\right)  \nonumber \\
	&=&\prod_{0\leq i<j\leq t-1} (e^{2Izu_i} - e^{2Izu_j}) \nonumber \\
	&=&(2I)^{n(n-1)/2}\prod_{0\leq i<j\leq t-1} \sin(u_iz - u_jz) 
\end{eqnarray}
and
\begin{eqnarray}
 \det\left(x_i^{-u_{j-1}}\right) 
	&=&(2I)^{n(n-1)/2}\prod_{0\leq i<j\leq t-1} \sin(u_iz - u_jz)  \nonumber\\
	&=&(2I)^{n(n-1)/2}(-1)^{t'}\prod_{k=1}^{t-1} \sin^{t-k}(kz) .
\end{eqnarray}
By the last three equations, we have
\begin{eqnarray*}
& &\prod_{m\geq 1} \Bigl((1-q^m)^{t-1}
\prod_{1\leq j<i\leq t} 
\Bigl(1-{\frac{e^{-2Iz(i-1)}}{e^{-2Iz(j-1)}}} q^{m}\Bigr)
\Bigl(1-{\frac{e^{-2Iz(j-1)}}{e^{-2Iz(i-1)}}} q^{m}\Bigr) \Bigr)\nonumber\\ 
& &=\prod_{m\geq 1} \Bigl((1-q^m)^{t-1}
\prod_{1\leq i \leq t-1} 
(1-e^{-2Iz(t-i)} q^m)^i
(1-e^{-2Iz(-t+i)} q^m)^i\Bigr) \nonumber\\
& &=
\frac{ (-1)^{t'}}
{ \prod_{k=1}^{t-1} \sin^{t-k}(kz) }
 \sum_{{\bf u}} 
q^{|\lambda|}    \prod_{0\leq i<j\leq t-1} \sin(u_iz - u_jz).
\end{eqnarray*}

We now make use of the following easy properties of the  $\sin$ function.
\begin{lemma}\label{thm.sin.properties}
Let $x,y, u_1, u_2, \ldots, u_n$ be complex numbers such that 
$u_1+u_2+\cdots +u_n=0$. Then
\begin{eqnarray}
\sin(x-y)\sin(x+y)=\sin^2(x) -\sin^2(y) ; \\
\prod_{1\leq i<j\leq n}\sin(u_i - u_j)									 
=
\prod_{1\leq i<j\leq n}\frac{e^{2u_iI}- e^{2u_jI}}{2I}.
\end{eqnarray}
\end{lemma}

Taking $\tau(k)=\sin(kz)$ in Equation (\ref{eqn.vcoding.evenodd}) 
and using Lemma \ref{thm.sin.properties} we obtain the following result.

\begin{lemma}\label{thm.t.integer}
For any complex number $p$ and any odd positive integer~$t$, we have
\begin{multline}\label{eqn.t.integer}
\sum_{\lambda\in T(t)} q^{|\lambda|}\prod_{h\in \setH(\lambda)}
	 \Bigl(1-\frac{ \sin^2(tz) }{\sin^2(hz)} \Bigr) \\
	 =
\prod_{m\geq 1} \Bigl((1-q^m)^{t-1}
\prod_{1\leq i \leq t-1} 
(1-e^{-2Iz(t-i)} q^m)^i
(1-e^{2Iz(t-i)} q^m)^i\Bigr),
\end{multline}
where $T(t)$ is the set of all $t$-core partitions.
\end{lemma}

We can work with the logarithm of the right-hand side of Equation~(\ref{eqn.t.integer})
to get 
\begin{eqnarray*}
& &\sum_k {\frac{-1}{k}} \sum_{m\geq 1} \bigl( (t-1) q^{mk}+
\sum_{i=1}^{t-1} 
i e^{-2Iz(t-i)k} q^{mk}+
i e^{2Iz(t-i)k} q^{mk}))\bigr) \\
&=&\sum_k {\frac{-q^k}{k(1-q^k)}}  \bigl( (t-1) +
\sum_{i=1}^{t-1} 
(i e^{-2Iz(t-i)k} +
i e^{2Iz(t-i)k} )\bigr) \\
&=&\sum_k {\frac{q^k}{k(1-q^k)}}  
\left(1-{\frac{e^{-2Iztk} + e^{2Iztk} -2 }{e^{-2Izk} + e^{2Izk} -2}}\right).
\end{eqnarray*}
Lemma~\ref{thm.t.integer} becomes the following lemma.
\begin{lemma}\label{thm.t.integer.exp}
For any complex numbers $z$ and any odd positive integer~$t$, we have
\begin{multline}\label{eqn.t.integer.exp}
 \sum_{\lambda\in T(t)} q^{|\lambda|}\prod_{h\in \setH(\lambda)}
	 \Bigl(1-{\frac{ \sin^2(tz) }{ \sin^2(hz) }} \Bigr)  \\
=
\exp\left(\sum_{k=1}^{\infty} {\frac{q^k}{k(1-q^k)}}  
\Bigl(1-\frac{\sin^2(tkz)}{\sin^2(kz)}\Bigr)\right).
\end{multline}
\end{lemma}
\medskip

\begin{proof}[Proof of Theorem \ref{thm.sin.H1}]
It is enough to prove that Equation~(\ref{eqn.poly})
is true for any complex numbers $z$ and $s$.
Let $n$ be a positive integer.
The coefficient $C_n(s)$ (resp.~$D_n(s)$) of $q^n$ on the left-hand side 
(resp.~right-hand side) of Equation~(\ref{eqn.poly}) is a polynomial
in $s$ of degree $2n$. 
For the proof of $C_n(s)=D_n(s)$, it suffices to find $2n+1$
explicit numerical values $s_0, s_1, \ldots, s_{2n}$ such that 
$C_n(s_i)= D_n(s_i)$ for $0\leq i\leq 2n$
by using the Lagrange interpolation formula.
The basic fact is that
\begin{equation*}
\prod_{h\in \setH(\lambda)}
	 \Bigl(s+{\frac{s^2-2s+1 }{4\sin^2(hz)}} \Bigr) =0 
\end{equation*}
for every partition $\lambda$ which is not a $t$-core (remember that
		$s=e^{2tz}$).
By comparing Theorem \ref{thm.sin.Ht} and Lemma \ref{thm.t.integer.exp} 
we see that Equation~\ref{eqn.poly}  is true
when $s=e^{2tz}$ for every odd integer~$t$, \emph{i.e.}~$C_n(e^{2tz})=D_n(e^{2tz})$.
This guarantees $C_n(s)=D_n(s)$ for every complex number $s$.  
\end{proof}

Recall the following result obtained in \cite{HanJi}.
\begin{thm}[Multiplication Theorem]
\label{thm.hanji}
If the series $f_\alpha(q)$  and the function $ \rho(h)$ satisfy
the relation
\begin{equation}
\sum_{\lambda \in \setP}q^{|\lambda|} \prod_{h \in \setH(\lambda)}\rho(\alpha
h)=f_\alpha(q),
\end{equation}
then, for any positive integer $r$, the following identity holds:
\begin{equation}
\sum_{\lambda\in\setP} q^{|\lambda|} x^{\#\setH_r(\lambda)}\prod_{h \in
\setH_r(\lambda)}\rho(h) = \left (f_r(xq^r)\right)^r\prod_{k\geq 1} {\frac{
(1-q^{rk} )^r
}{(1-q^k)}}.
\end{equation}
\end{thm}
This last result can be used as a transition from Theorem~\ref{thm.sin.H1} to Theorem~\ref{thm.sin.Ht}.
\begin{proof}[Proof of Theorem \ref{thm.sin.Ht}]
Let $\rho(h)=1-\sin^2(tz)/\sin^2(hz)$ in Theorem \ref{thm.hanji}. We get
\begin{eqnarray*}
f_\alpha(q)
&=& \sum_{\lambda \in \setP}q^{|\lambda|} 
\prod_{h \in \setH(\lambda)}\rho(\alpha h) \\
&=& \sum_{\lambda \in \setP}q^{|\lambda|} 
\prod_{h \in \setH(\lambda)}\Bigl(1-\frac{\sin^2(tz)}{\sin^2(\alpha hz)}\Bigr)
	\\
&	\overset{\text{Thm \ref{thm.sin.H1}}}{=}&
\exp\Bigl(\sum_{k=1}^{\infty} {\frac{q^k}{k(1-q^k)}}  
\bigl(1-\frac{\sin^2(tkz)}{\sin^2(\alpha kz)}\bigr)\Bigr).
\end{eqnarray*}
Hence,
\begin{eqnarray*}
 & &\sum_\lambda q^{|\lambda|}\prod_{h\in \setH_r(\lambda)}
	 \Bigl(1-{\frac{ \sin^2(tz) }{ \sin^2(hz) }} \Bigr)   \\
	&=& 
\exp\Bigl(r\sum_{k=1}^{\infty} {\frac{q^{rk}}{k(1-q^{rk})}}  
\bigl(1-\frac{\sin^2(tkz)}{\sin^2(r kz)}\bigr)\Bigr) 
	\prod_{k\geq 1} {\frac{ (1-q^{rk} )^r }{(1-q^k)}} \\
&=&
\exp\sum_{k=1}^{\infty} \Bigl({\frac{q^k}{k(1-q^k)}}  
-\frac{rq^{rk}}{k(1-q^{rk})}\frac{\sin^2(tkz)}{\sin^2(rkz)}\Bigr).
\end{eqnarray*}
\end{proof}


\section{Multiset hook-content formula} 
\label{sec.content}
In this section we establish a multiset hook-content formula.
Let $s_\lambda$ be the Schur function corresponding to the partition 
$\lambda$ (see \cite[p.40]{macdonald}, \cite[p.308]{stanley}, 
\cite[p.8]{lascoux}). 
Recall the following classical hook-content formula (\cite[p.374]{stanley},
\cite{robinson}).

\begin{thm}\label{thm.hookcontent.classical}
For any partition $\lambda$ and positive integer $n$ we have
\begin{equation}\label{eqn.hookcontent.classical}
s_\lambda(1,p,p^2,\ldots, p^{n-1})=p^{b(\lambda)} 
\prod_{\square\in\lambda} \frac{1-p^{n+c_\smallsquare}}{ 1-p^{h_\smallsquare}},
\end{equation}
where $b(\lambda)=\sum_i (i-1) \lambda_i$ and 
$c_\smallsquare=j-i$ if $\smallsquare\in\lambda$ occurs on the $i^\text{th}$ row and $j^\text{th}$ column of the diagram of $\lambda$.
\end{thm}
We now state a Theorem that provides an alternative approach to the left-hand side of Equation~(\ref{eqn.vcoding}). 
\begin{thm}
\label{thm.content} Let $t$ be a positive integer.
There is a bijection $\psi_t: \lambda \mapsto \mu$
which maps $t$-cores
onto the set of all partitions $\mu$ of length at most $t-1$ 
such that $\{	\mu_i-i\mod t : i=1, \ldots,t\}=\{0,1,\ldots, t-1\} $. 
Moreover, given any $\tau$ from $\mathbb{Z}$ to a field $F$, we have
\begin{eqnarray}
\label{eqn.content.size}
|\lambda|  	&=& -|\mu| \frac{|\mu|+t+t^2}{2t^2} +  \sum_{i=1}^t \frac{\mu_i^2 + 2(t+1-i) \mu_i}{2t}
\end{eqnarray}
and
\begin{eqnarray}
\prod_{\square\in\lambda} {\frac{\tau(h_\square-t) \tau(h_\square+t) }{\tau(h_\square)^2}}
&=&
\prod_{i=1}^{t-1} \Bigl({\frac{\tau(-i)}{\tau(i)}}\Bigr)^{\beta_i(\lambda)}
\prod_{\square\in\mu} 
{\frac{ \tau(t+ c_\square) }{\tau(h_\square)
}}.
\label{eqn.content}
\end{eqnarray}
\end{thm}
In fact, Theorem \ref{thm.sin.Ht} can also be proved by using 
the multiset hook-content formula (Theorem \ref{thm.content})
and the  hook-content formula (Theorem \ref{thm.hookcontent.classical}).
Conversely, the hook-content formula 
(\ref{eqn.hookcontent.classical}) can be derived by using the multiset hook-content formula (\ref{eqn.content}) and Theorem \ref{thm.sin.Ht}. This justifies the name of this section.

\begin{proof}[Proof of Theorem~\ref{thm.content}]
We give an explicit description of $\mu=\psi_t(\lambda)$. Let $a= M_2(\lambda) = -\min V(\lambda)$ and $\vec{V}=\phi_t(\lambda)$ 
be the $V_t$-coding of $\lambda$.  
We also set $\mu = \vec{\mu} := (\vec{V}_i -t+i+a: i \in \{1,\cdots, t\})$ to be the (ordered) parts of a partition, trailing with at least one zero.
The temporary arrow notation for vectors is meant to emphasize that it is now sorted by decreasing order.
The partition $\mu$ may be rewritten as
$\vec{\mu} =  \vec{V} +a \vec{1}-\vec{b}$ where
$\vec{b} = (t-1,\cdots,0)$ and $\vec{1} = (1,\cdots,1)$. 
We also know that 
$$\frac{1}{2t} \vec{V} \cdot \vec{V} = |\lambda| + \left(\frac{t^2-1}{24}\right)$$ (by Equation~(\ref{eqn.sum.square})), and that $\vec{V}\cdot\vec{1}=0$ (by Equation~(\ref{eqn.sum.zero})). The statement to be proved is then 
$$
|\lambda| = -(\vec{\mu}\cdot\vec{1}) \frac{\vec{\mu}\cdot\vec{1}+t+t^2}{2t^2}+\frac{\vec{\mu}\cdot\vec{\mu}}{2t} +\frac{1}{t}\left(\vec{1}+\vec{b}\right)\cdot \vec{\mu}.
$$
But this follows readily from 
$\vec{b}\cdot\vec{b} = 
{t(t-1)(2t-1)}/{6}$ and
$\vec{b}\cdot\vec{1} = {t(t-1)}/{2}$. 
\medskip

We now move on to Equation~(\ref{eqn.content}).
Define $\tau!(i) = \prod_{j=1}^i \tau(j)$. 
On the other hand, we have
\begin{eqnarray}
\prod_{\square \in \mu} {\tau(c_\square+t)} &=& \frac{\prod_{i=1}^t \tau!(\mu_i+t-i)}{\prod_{i=1}^{t-1} \tau(i)^{t-i} }.
\label{eqn.mu}
\end{eqnarray}
The partition $\mu$ can be viewed in the exploted tableau by the set
$ \{(x,y)\in V_1\times ([M_2,-M_1]\setminus -V_1) \mid x+y>0\}$. Hence
\begin{multline}
\prod_{\square \in \mu} \tau(h_\square) = \prod_{\substack{(x,y) \in V_1\times ([M_2,-M_1]\setminus -V_1)\\ x+y > 0 }} \tau(x+y)\\ = \prod_{x \in V_1} \prod_{\substack{y \in ([M_2,-M_1]\setminus -V_1)\\ x+y > 0 }} \tau(x+y) = \frac{\prod_{x \in V_1}  \tau!(M_2+x)}{\prod_{0\le i,j \le t-1} \tau(v_i - v_j)}.
\label{eqn.tau.factorial}
\end{multline}
Equations~(\ref{eqn.mu}) and (\ref{eqn.tau.factorial}), together with Theorem \ref{thm.vcoding} suffice to establish~(\ref{eqn.content}).
\end{proof}


\bibliographystyle{\jobname}

\bibliography{\jobname}

\end{document}